\documentclass[11pt,amsmath,amsthm,amsfonts,amssymb,amscd]{article}
\usepackage{epsfig}
\usepackage[latin1]{inputenc}
\usepackage{amsthm}
\usepackage{amsmath}
\usepackage{graphicx}
\usepackage{amssymb}
\usepackage{amscd}

\font \Bbbten=msbm10 \font \Bbbsev=msbm7 \font \Bbbfiv=msbm5
\newfam\Bbbfam \def \Bbb{\fam\Bbbfam\Bbbten}\textfont\Bbbfam = \Bbbten
\scriptfont\Bbbfam=\Bbbsev \scriptscriptfont\Bbbfam=\Bbbfiv
\newcommand{\N}{\mbox{$I\!\!N$}}
\newcommand{\Z}{\mbox{$Z\!\!\!Z$}}

\newcommand{\R}{\mbox{$I\!\!R$}}

\newcommand{\T}{{\Bbb T}}

\newcommand{\dist}{{\rm dist}}

\newcommand{\cita}[7]{{\sc #1, }{\it #2, }{\small #3, {\bf #4 } (#5), p.
#6-#7.}}
\newcommand{\cit}[5]{{\sc #1, }{\it #2, }{\small #3, { #4 } #5.}}

\theoremstyle{plain}
\newtheorem{thm}{Theorem}[section]
\newtheorem{rk}[thm]{Remark}

\newtheorem{Df}{Definition}[section]
\newtheorem{Teo}{Theorem}[section]
\newtheorem{Lem}[Teo]{Lemma}

\newtheorem{Obs}[Teo]{Remark}

\title{ Lyapunov exponents for expansive homeomorphisms}
\author{
M. J. Pacifico\footnote{partially supported by
FAPERJ},
J. L. Vieitez\footnote{partially supported by Grupo de Investigaci\'on "Sistemas Din\'amicos" CSIC (Universidad de la
Rep\'ublica), SNI-ANII, PEDECIBA, Uruguay}}
\date{\today}

\begin{document}
\maketitle
\begin{abstract}
Let $(M,\rho)$ be a compact metric space and $f:M\to M$ an expansive homeomorphism.
We define Lyapunov exponents $\Lambda(f,\mu)_{max}$ and $\lambda(f,\mu)_{min}$ for an $f$-invariant measure $\mu$.
When $\Lambda(f,\mu)_{max} > 0$ and $\lambda(f,\mu)_{min} < 0$
can be interpreted as a weak form of hyperbolicity for $f$.
We prove that if $M$ is a Peano space then there is $\gamma>0$ such that
$\Lambda(f,\mu)_{max}>\gamma$ and $\lambda(f,\mu)_{min}<-\gamma$.
We also show that the hypothesis that $M$ is a Peano space is necessary to obtain the maximal Lyapunov exponent positive and the minimal Lyapunov exponent negative. Moreover we define Lyapunov exponents for $K$, a compact $f$-invariant subset of $M$ and prove that if the maximal Lyapunov exponent of $K$ is negative then $K$ is an attractor.
When $f$ is a diffeomorphism on a compact manifold, these Lyapunov exponents coincide with the usual ones.
\end{abstract}

\section{Introduction}

In the study of differentiable dynamics an indication of chaos is given by the so called Lyapunov exponents or characteristic exponents.
 Their use in Physics was initially based on the following considerations which in fact goes in the opposite direction: trying to ensure stability of motions.
 Let the differential equation $\dot x =F(x)$ define an autonomous dynamical system where $F:\Omega\subset \R^n\to\R^n$ is $C^1$ and $\Omega$ is open.
For $x_0\in\Omega$ consider the solution $\varphi(t,x_0)$ of the initial value problem
\begin{equation} \label{eq1}
\left\{\begin{array}{c}
\dot x=F(x) \\
x(0)=x_0
\end{array}\right.
\end{equation}
 Assume that all solutions of (\ref{eq1}) with initial condition $x_1$ in a neighborhood of $x_0$ do exist for $t\in [0,+\infty)$.
 An experimenter will probably have an error in the measurements for initial data slightly altered and the initial data will be
$ x_1=x_0+y$ instead of $x_0$ where $y$ is the error in the measurement that is supposed small.
 The dynamical behavior of the nearby solution can be described approximately by the linearization of $\dot x=F(x)$, that is, by the linear system of differential equations  $\dot y=DF_x(\varphi(t,x_0))y$ where $\varphi(t,x_0)$ is supposed to be the ''correct'' solution. If for all small $y$ the solution $\widetilde \varphi(t,y)$ of the system $\dot y=DF_x(\varphi(t,x_0))y$ tends to zero when $t\to+\infty$  then this is seen as an indication of (asymptotic) stability of the motion.
 A way to capture this is given by the limit $\chi_{x_0}(y)=\lim_{t\to+\infty}\frac{1}{t}\log(\|\widetilde \varphi(t,y)\|)$ whenever this limit exists. In this case, this limit gives information about exponential convergence (if $\chi_{x_0}(y)<0$ for all $y$ small) or divergence (total instability if $\chi_{x_0}(y)>0$ for all $y$ small) of trajectories with respect to the initial data problem.
 If the limit does not exist we instead can consider the $\limsup$ if we want to capture by this means any kind of exponential divergence.

 In the discrete case, i.e., $t=n\in\Z$, when a $C^1$-dynamical system is given by a differentiable map $f:M\to M$ where $M$ is a compact smooth manifold, the Lyapunov exponent is given  for $x\in M$ and $v\in T_xM$ by $\chi(x,v)= \limsup_{n\to\infty}\frac{1}{n}\log(\|Df^n_x(v)\|)$. Here $v$ takes the place of the ''error'' $y$ via the inverse of the exponential map $\exp_x:T_xM\to M$.

\medbreak

  One problem with this approach is that in various situations we cannot assume that the system given by $f$ is differentiable and therefore the computations roughly described above have no sense.
  Moreover in several cases an experimenter has a collection of data  indicating  that the map $f$ is continuous and even differentiable, but has not enough data to obtain an approximation of the  differential map $Df$.
  So it seems of interest to introduce some kind of Lyapunov exponents for the case of a continuous dynamical system. This has been done by Barreira and Silva (\cite{BS}) for continuous maps $f:\R^n\to\R^n$, and by Kifer  (\cite{Kif}) for the case $f:X\to X$ where $X$ is a compact metric space.
  We will address the problem of defining Lyapunov exponents for an expansive homeomorphism $f$ on a compact metric space $(X,\dist)$ using similar techniques as those developed in \cite{BS,Kif}. Under certain conditions about the topology of the space $X$ where $f$ acts we obtain that the Lyapunov exponents are different from zero, indicating that $f$ presents a chaotic dynamics.

\medbreak


 \section{Lyapunov exponents for expansive homeomorphisms.}

Let $f:M\to M$ be a homeomorphism defined on a compact metric space $(M,\dist)$.
Following \cite{Kif} we define maximal and minimal Lyapunov exponents with respect to the distance $\dist:M\times M\to \R$ for a homeomorphism $f$.
Assume $M$ has no isolated points.

Let $N\in\N$ and define $$B^*_x(\delta,N)=\{y\in M\setminus \{x\}\,:\, \dist(f^j(x),f^j(y))\leq\delta,\, \forall\, j=0,1,\ldots,N\}\, .$$
If $N<0$ define $B^*_x(\delta,N)=\{y\in M\setminus \{x\}\,:\, \dist(f^j(x),f^j(y))\leq\delta,\, \forall\, j=N,N+1,\ldots,-1,0\}$.
For $n\in\Z$, $\delta>0$ and $x\in M$  define
$$A_\delta(x,n)=\sup_{y\in B^*_x(\delta,n)} \left\{\frac{\dist(f^n(x),f^n(y))}{\dist(x,y)}\right\}$$
and
$$a_\delta(x,n)=\inf_{y\in B^*_x(\delta,n)} \left\{\frac{\dist(f^n(x),f^n(y))}{\dist(x,y)}\right\}\, .$$

\begin{rk} Note that $A_\delta(x,n)$ and $a_\delta(x,n)$ can be interpreted as the maximal, respectively the minimal  distortion of $f$ on $B^*_x(\delta,N)$.
\end{rk}

Let $\mu$ be a Borel $f$-invariant probability measure and assume that there is $\varepsilon_0>0$ such that for all $0<\delta<\varepsilon_0$ it holds that

\begin{equation} \label{condition}
\sup_{n\in\Z\backslash \{0\}}\frac{1}{|n|}\int_M|\log(A_\delta(x,n))|\mu(dx)<\infty
\end{equation}
\begin{equation*}
\left|\inf_{n\in\Z\backslash \{0\}}\frac{1}{|n|}\int_M|\log(a_\delta(x,n))|\mu(dx)\right|<\infty\, .
\end{equation*}

In this case we define
$$\Lambda^+_\delta(x)=\limsup_{n\to+\infty}\frac{1}{n}\log (A_\delta(x,n))\; \mbox{ and }\;
\lambda^+_\delta(x)=\limsup_{n\to+\infty}\frac{1}{n}\log (a_\delta(x,n))$$
and for $n<0$
$$\Lambda^-_\delta(x)=-\limsup_{n\to-\infty}\frac{1}{n}\log (A_\delta(x,n))\; \mbox{ and }\;
\lambda^-_\delta(x)=-\limsup_{n\to-\infty}\frac{1}{n}\log (a_\delta(x,n))\, .$$

The following result is proved in \cite[Theorem 1]{Kif}.

\begin{Teo} \label{Kifer}
For $x\in M$ $\mu$ a.e. it holds that
the limits

$$\Lambda^+_\delta(x)=\lim_{n\to+\infty}\frac{1}{n}\log (A_\delta(x,n))\, , \quad
\lambda^+_\delta(x)=\lim_{n\to+\infty}\frac{1}{n}\log (a_\delta(x,n)),$$

$$\Lambda^-_\delta(x)=-\lim_{n\to-\infty}\frac{1}{n}\log (A_\delta(x,n))\, , \quad
 \lambda^-_\delta(x)=-\lim_{n\to-\infty}\frac{1}{n}\log (a_\delta(x,n))$$
do exist.
Moreover, $\Lambda^+_\delta(x)=-\lambda^-_\delta(x)$ and $\lambda^+_\delta(x)=-\Lambda^-_\delta(x)$
and $\Lambda^+_\delta(x)$ and $\lambda^+_\delta(x)$ are $f$-invariant $\mu$ a.e. . Similarly for $\Lambda^-_\delta(x)$ and $\lambda^-_\delta(x)$.
\end{Teo}

Since we are assuming that (\ref{condition}) is valid and $A_\delta(x,n)$ decreases when $\delta$ decreases to zero the limit
$\Lambda^+(x)=\lim_{\delta\to 0}\Lambda^+_\delta(x) \mbox{ exists}\, .$
Analogously since $a_\delta(x,n)$ increases when $\delta$ decreases the limit
$\lambda^+(x)=\lim_{\delta\to 0}\lambda^+_\delta(x) \, \mbox{ exists}\, .$
Similarly, for $\mu$ a.e., there exist $\Lambda^-(x)$ and $\lambda^-(x)$.
Thus we introduce the following definition
\begin{Df}
We define the Lyapunov exponents for $f$ at $x\in M$ by
$$\Lambda^+(x)=\lim_{\delta\to 0}\Lambda^+_\delta(x),\,\; \lambda^+(x)=\lim_{\delta\to 0}\lambda^+_\delta(x)$$
and similarly for $\Lambda^-(x)$ and $\lambda^-(x)$. As proved above these quantities exist $\mu$ a.e. and are $f$-invariant.
\end{Df}

 Next we compute these Lyapunov exponents for  an expansive homeomorphism. To do so,
 let us recall that a homeomorphism $f:X\to X$, $X$ a compact metric space, is expansive
if there exists $\alpha>0$ such that for all $x,y\in X$ if $x\neq y$
then there is $n\in \Z$ such that $\dist(f^n(x),f^n(y))> \alpha$.
We will obtain those Lyapunov exponents with respect to a hyperbolic metric adapted to the expansive homeomorphism, given by \cite[Theorem 5.1]{Ft}:
\begin{Teo} \label{Fathi}
Let  $f:M\to M$ be an
expansive homeomorphism of the compact metric space $(M,\dist)$. Then there exists a metric
$d:M\times M\to\R$ on $M$, defining the same topology as $\dist$, and numbers $k > 1$, $\varepsilon_0 > 0$ such that:
$$\forall x, y \in M, \max\{d(f(x),f(y)), d(f^{-1}(x),f^{-1}(y))\}\geq \min \{kd(x,y),\varepsilon_0\}\, .$$
Moreover, both $f$ and $f^{-1}$ are Lipschitz for $d$.
\end{Teo}
\begin{Obs}
The existence of an expansive homeomorphism on $M$ implies that the topological dimension of $M$ is finite, see\cite{Ma2}.
\end{Obs}
To define $\Lambda^{\pm}(x)$ and $\lambda^{\pm}(x)$ for $x\in M$, we need to show that condition (\ref{condition}) is fulfilled. To this end, we first verify the following

\begin{Lem}\label{previous}
Let $\mu$ be a Borel probability measure invariant by $f:M\to M$.
 If $f$ is expansive and $d$ is the distance defined by Theorem \ref{Fathi} then
 $$\sup_{n\in\Z\backslash\{0\}}\frac{1}{|n|}\int_M|\log(A_\delta(x,n))|\mu(dx)<\infty\, $$ and $$\left|\inf_{n\in\Z\backslash\{0\}}\frac{1}{|n|}\int_M|\log(a_\delta(x,n))|\mu(dx)\right|<\infty.$$
\end{Lem}
\begin{proof}
By Theorem \ref{Fathi} $f$ and $f^{-1}$ are Lipschitz with respect to the metric $d$, i.e., there is a constant $K> 1$ such that
$$\forall x,y \in M:\!x\neq y,\quad \frac{d(f(x),f(y))}{d(x,y)}\leq K\mbox{ and }
\frac{d(f^{-1}(x),f^{-1}(y))}{d(x,y)}\leq K\, .$$
From the last inequality it follows that $\forall \, x,y\in \!M, \, x\neq \!y, \frac{d(f(x),f(y)}{d(x,y)}\geq \frac{1}{K}$.

\noindent
Thus, $\sup_{x,y\in M, x\neq y}\frac{d(f^n(x),f^n(y))}{d(x,y)}\leq K^{|n|}$ for all $n\in\Z$. Hence $\log(|A_\delta(x,n)|)\leq |n|\log(K)$ for all $\delta>0$, $x\in M$ and $n\in\Z$. Therefore
$$ \sup_{n\in\Z\backslash\{0\}}\frac{1}{|n|}\int_M|\log(A_\delta(x,n))|\mu(dx)<\infty\;\mbox{ and condition (}\ref{condition}\mbox{) holds} .$$

\noindent Moreover since
$$a_\delta(x,n)=\inf_{y\in B^*_x(\delta,n)} \{d(f^n(x),f^n(y))/d(x,y)\}=\left(\sup_{y\in B^*_x(\delta,n)}\{d(x,y)/d(f^n(x),f^n(y))\}\right)^{-1} =$$    $$=\frac{1}{A_\delta(f^n(x),-n)}$$
and $\mu$ is $f$-invariant we also have that
$$\left|\inf_{n\in\Z\backslash\{0\}}\frac{1}{|n|}\int_M|\log(a_\delta(x,n))|\mu(dx)\right|<\infty\, .$$
The proof is complete.
\end{proof}

Note that Lemma \ref{previous} and Theorem \ref{Kifer} imply that for any $f$-invariant measure $\mu$ the numbers $\Lambda^+(x)$, $\lambda^+(x)$, $\Lambda^-(x)$,
$\lambda^-(x)$ do exist $\mu$ a.e. and are $f$ invariant.

Recall that $M$ is a Peano space if it is connected, locally connected compact metric space. Next we give a positive lower bound of $\Lambda^+(x)$ and a negative lower bound of $\lambda^+(x)$ for an expansive homeomorphism $f:M\to M$ defined on a compact Peano space. As remarked above this can be
interpreted as a weak kind of hyperbolicity condition.

\begin{Teo}\label{Peano}
Let $(M,d)$ be a compact connected and locally connected metric space.
Let $f:M\to M$ be an expansive homeomorphism and $\gamma=\log(k)$ where
 $k>1$ is the constant given by Theorem \ref{Fathi}. Then for all $x\in M$ it holds $\Lambda^+(x)\geq\gamma$ and $\lambda^+(x)\leq-\gamma$.
\end{Teo}
\begin{proof}
Given a point $x\in M$ there is $y\in M\backslash \{x\}$ close to $x$ such that $d(f(x),f(y))\geq kd(x,y)$ where $d(\cdot,\cdot)$ is the distance given by Theorem \ref{Fathi}.
Otherwise, by the mentioned theorem, for some $\delta>0$ and every point $y\in B(x,\delta)$ we have $d(f(x),f(y))< kd(x,y)$ and therefore
for all $y\in B(x,\delta)$ it holds that $d(f^{-1}(x),f^{-1}(y)\geq kd(x,y)$. Thus $B(f^{-1}(x),\delta)\subset f^{-1}(B(x,\delta))$.
Moreover we also have for all $y\in B(f^{-1}(x),\delta)$ that $d(f^{-2}(x),f^{-1}(y))\geq kd(f^{-1}(x),y)$. For we already know that for every point $z\in B(f^{-1}(x),\delta)$ it holds that $d(f(z),f(f^{-1}(x)))\leq \frac{1}{k}d(f^{-1}(x),z)$. By induction we obtain a sequence of balls
$B(f^{-n}(x),\delta)$ such that for all $y\in B(f^{-n}(x),\delta)$ we have $d(f^{-n-1}(x),f^{-1}(y))\geq kd(f^{-n}(x),y)$. Let $z$ be an $\alpha$-limit point of the sequence $\{f^{-n}(x)\}$. Then $z$ is a Lyapunov stable point of $f$ contradicting that there are no such points if $f:M\to M$ is expansive and $M$ is compact connected and locally connected, see \cite[Proposition 2.7]{Le}.
Hence, for every $\delta>0$ there is $y\in B(x,\delta)\backslash \{x\}$ such that $d(f(x),f(y))\geq kd(x,y)$.

Given $n>0$ let $\delta>0$ be so small that in $B^*_x(\delta,n)=\{y\in M\setminus \{x\}\,:\, d(f^j(x),f^j(y))\leq\delta\, \forall\, j=0,1,\ldots,n\}$ we have $d(f^j(x)f^j(y))\leq \varepsilon_0$ for all $j=1,2\ldots, n$ where $\varepsilon_0>0$ is given by Theorem \ref{Fathi}.
As a consequence of the previous paragraph there is a point $y\in B^*_x(\delta,n)$ such that $d(f^j(x),f^j(y)) \geq k^jd(x,y)$ for all $j=1,\ldots,n$.
 Therefore
 $$A_\delta(x,n)=\sup_{y\in B^*_x(\delta,n)} \{d(f^n(x),f^n(y))/d(x,y)\}\geq k^n,$$
implying that $\Lambda^+_\delta(x)=\lim_{n\to+\infty}\frac{1}{n}\log (A_\delta(x,n))\geq\log (k)=\gamma>0$.

 Similarly $\lambda^-_\delta(x)\leq -\log(k)=-\gamma$.
 Since this is valid for any small $\delta>0$, letting $\delta\to 0$ we obtain that $\Lambda^+(x)\geq\gamma$ and $\lambda^+(x)\leq-\gamma$ finishing the proof.

\end{proof}

\begin{Obs}
 When $f:M\to M$ is a diffeomorphism on a compact manifold these Lyapunov exponents coincides with the usual ones, see \cite{BS,Kif}.
\end{Obs}

Next we construct an example, inspired in \cite{RR}, of an expansive homeomorphism defined on a compact connected metric space exhibiting Lyapunov stable points, showing that the hypothesis of locally connectedness can not be negligible in Theorem \ref{Peano}.

\begin{Teo}\label{example} The hypotheses of local connectedness cannot be
negligible in Theorem~ \ref{Peano}.
\end{Teo}
\begin{proof}
Let $f_A: \T^2 \to \T^2$ be the Anosov map in the two-torus $\T^2$ induced by the matrix
$
A =\left(
\begin{array}{ll}
2
& 3 \\
3
& 5
\end{array}
\right).
$
Let $p$ the fixed point of $f_A$ corresponding to the origin and $v_p,$ be an eigenvector associated to the smallest eigenvalue $\lambda=\frac{7-3\sqrt{5}}{2} < 1$ of $A$. Fix, for instance, $v_p=\left(1,\frac{1-\sqrt{5}}{2}\right)$.
Since the coordinates of $v_p$ are not rational numbers the natural projection of $\{t v_p, t \in \R\}$ into $\T^2$ is dense in $\T^2$ and corresponds to the stable manifold $W^s(p)$ of the fixed point $p$.

Identify $\R^2$ with the plane $Oxy$ and consider the point $q=(0,0,\epsilon)$ with $0<\epsilon<1$. Observe that $v_p=\left(1,\frac{1-\sqrt{5}}{2},0 \right)$ is parallel to $Oxy$.
Let $\gamma\subset \R^3$ be the curve given by the equation
$$\gamma(t)=tv_p+(0,0,\frac{\epsilon}{t^2+1}), \,\, t\in\R\, .$$
Then $\gamma$ is asymptotic to the straight line of $\R^3$ given by $\{(0,0,0)+tv_p,\,t\in\R\}\subset Oxy$.
Define an extension $\widetilde A$ of $A$ to $Oxy\cup \gamma$ in the following way: for points $(x,y,0)\in Oxy$ we define $\widetilde A (x,y,0)=A(x,y)$ and for $\gamma$ we define $\widetilde A (\gamma(t))=\lambda tv_p+(0,0,\frac{\epsilon}{(\lambda t)^2+1}) \, t\in\R\, .$
Then $\widetilde A$ has the inverse $A^{-1}$ for points $(x,y,0)$ and for points in $\gamma$ given by
$$\widetilde A^{-1} (\gamma(t))=\frac{t}{\lambda}v_p+\left(0,0,\frac{\epsilon}{\left(\frac{t}{\lambda}\right)^2+1}\right) \, t\in \R\, .$$
Observe that $\gamma(0)=q=\widetilde A(\gamma(0))$.
On its turn, factoring out the integer lattice
$\Z \times \Z\times \{0\}$ in $\R^3$, we get a homeomorphism $f: \T^2 \cup \mathcal{H} \to T^2 \cup \mathcal{H}$ where $\mathcal{H}$ is the image of $\gamma$ on the quotient space.
As $\mathcal{H}$ is a copy of $W^{s}(p)$ and for $t\to \infty$ the distance of $\gamma(t)$ to $Oxy$ goes to 0,
$\mathcal{H}$ is a curve asymptotic to $\T^2$, we obtain that
$X=\T^2 \cup \mathcal{H}$ is compact and connected.
We then define a dynamics in $X$ in the following way: in $\T^2$ is the dynamics induced by $A$ and in
$\mathcal{H}$ is the dynamics of $W^{s}(p)$. It turns out that this dynamics in $X$ is expansive.
But the points in $\mathcal{H}$ are stable. In particular, so is the point $q$, implying that
 $q$ has a unique Lyapunov exponent (as it occurs for any point of $\mathcal{H}$), which is strictly less than zero, finishing the proof.
\end{proof}

%
%
%
%

\section{ Compact invariant subsets.}
In this section we extend the definition of Lyapunov exponents 
for compact $f$-invariant sets  of a homeomorphism defined on a Peano space. The goal is to proof that if the maximal Lyapunov exponent of $K$, a compact invariant set, is strictly negative then $K$ is an attractor.

Let $M$ be a (non trivial) compact Peano space and $f:M\to M$ a homeomorphism.
For $A\subset M$, $A\neq\emptyset$, and $x\in M$ we define $\dist(x,A)=\inf\{\dist(x,y)\,:\, y\in A\}$.

\noindent Let $K\subset M$ be a compact $f$-invariant subset of $M$, i.e., $f(K)= K$.
For $N\in\N$ define $$B^*_K(\delta,N)=\{y\in M\setminus K\,:\, \dist(f^j(y),K)\leq\delta,\, \forall\, j=0,1,\ldots,N\}\, .$$
If $N<0$ define $B^*_K(\delta,N)=\{y\in M\setminus K\,:\, \dist(f^j(y),K)\leq\delta,\, \forall\, j=N,N+1,\ldots,-1,0\}$.
For $n\in\Z$ and $\delta>0$ let us define
$$A_\delta(K,n)=\sup_{y\in B^*_K(\delta,n)} \left\{\frac{\dist(K,f^n(y))}{\dist(K,y)}\right\}$$
and
$$a_\delta(K,n)=\inf_{y\in B^*_K(\delta,n)} \left\{\frac{\dist(K,f^n(y))}{\dist(K,y)}\right\}\, .$$
Let us also define
$$\Lambda^+_\delta(K)=\limsup_{n\to+\infty}\frac{1}{n}\log (A_\delta(K,n))\; \mbox{ and }\;
\lambda^+_\delta(K)=\limsup_{n\to+\infty}\frac{1}{n}\log (a_\delta(K,n))$$
and for $n<0$
$$\Lambda^-_\delta(K)=-\limsup_{n\to-\infty}\frac{1}{n}\log (A_\delta(K,n))\; \mbox{ and }\;
\lambda^-_\delta(K)=-\limsup_{n\to-\infty}\frac{1}{n}\log (a_\delta(K,n))\, .$$

Since $f^n(K)=K$ and $B^*_K(\delta,n+k)\subset B^*_K(\delta,n)$, if $k\geq 0$, it holds that
$$A_\delta(K,n+k)=\sup_{y\in B^*_K(\delta,n+k)} \{\dist(K,f^{n+k}(y))/d(K,y)\}= $$
$$\sup_{y\in B^*_K(\delta,n+k)} \left\{\left(\frac{\dist(K,f^{n}(y))}{d(K,y)}\right)\cdot\left(\frac{\dist(K,f^{n+k}(y))}{d(K,f^n(y))}\right)\right\}\leq$$
$$\sup_{y\in B^*_K(\delta,n)}\! \{\dist(K,f^n(y))/d(K,y)\}\cdot \!\!\!\!\sup_{z\in B^*_{K}(\delta,k)} \!\{\dist(K,f^k(z))/d(K,z)\}=$$
$$=A_\delta(K,n)\cdot A_\delta(K,k). $$
Therefore letting $Y(\delta,K,n)=\log (A_\delta(K,n))$ we obtain a subadditive function and there is the limit $\Lambda^+(K,\delta)$ of
$\frac{1}{n}\log (A_\delta(K,n))$ for $n\to+\infty$.
Since $\log (A_\delta(K,n))$ is monotone in $\delta$ there exists $\Lambda^+(K)=\lim_{\delta\to 0}\Lambda^+(K,\delta)$.
Similarly there exist the limits $\lambda^+(K)= \lim_{\delta\to 0} \lambda^+_\delta(K)$, $\Lambda^-(K)=\lim_{\delta\to 0}\Lambda^-(K,\delta)$ and $\lambda^-(K)=\lim_{\delta\to 0} \lambda^-_\delta(K)$.
Thus, we introduce the definition below.
\begin{Df}
Let $f: M\to M$ be an expansive homeomorphism defined on a Peano space $M$. Given a compact, $f$-invariant set $K\subset M$, we define the Lyapunov exponents of $K$ by
$$\Lambda^+(K)=\lim_{\delta\to 0}\Lambda^+(K,\delta),\,\; \lambda^+(K)=\lim_{\delta\to 0}\lambda^+_\delta(K,\delta)$$
and similarly for $\Lambda^-(K)$ and $\lambda^-(K)$. 
\end{Df}

As for the case of a point $x\in M$ it can be proved that

\begin{Teo}
$\Lambda^+_\delta(K)=-\lambda^-_\delta(K)$ and $\lambda^+_\delta(K)=-\Lambda^-_\delta(K)$.
\end{Teo}
\begin{proof}
Indeed, we have that
$$a_\delta(K,n)=\inf_{y\in B^*_K(\delta,n)} \left\{\frac{\dist(K,f^n(y))}{\dist(K,y)}\right\}=$$
$$\left(\sup_{ y\in B^*_K(\delta,n)}\left\{\frac{\dist(K,y)}{\dist(K,f^n(y))}\right\}\right)^{-1}=$$
$$\left(\sup_{ y\in B^*_K(\delta,-n)}\left\{\frac{\dist(K,f^{-n}(y))}{\dist(K,y)}\right\}\right)^{-1}=A_\delta^{-1}(K,-n)\, .$$
Therefore $\lambda^+_\delta(K)=\lim_{n\to+\infty}\frac{1}{n}\log (a_\delta(K,n))=
-\lim_{n\to-\infty}\frac{1}{|n|}\log(A_\delta(K,-n))=-\Lambda^-_\delta(K)$.
Similarly it can be proved that $\Lambda^+_\delta(K)=-\lambda^-_\delta(K)$.
\end{proof}


Given a compact invariant set $K\subset M$, we say that $K$ is an {\em{attractor}} if
there is a neighborhood $U$ of $K$ such that if $y\in U$ then
$\lim_{n\to +\infty}\dist(f^n(y),K)=0$. Analously,  $K$ is a {\em{repeller }} if
$y\in U$ then
$\lim_{n\to -\infty}\dist(f^n(y),K)=0$.

\begin{Teo}
Let $M$ be a compact Peano space and $K \subset M$ be an invariant compact set. If $\Lambda^+(K)<0$ then $K$ is an attractor.
Analogously if $\lambda^-(K)>0$ then $K$ is a repeller.
\end{Teo}
\begin{proof}
Since  $\Lambda^+(K)=\lim_{\delta\to 0}\Lambda^+(K,\delta)<0$ there is $\delta_0>0$ such that for all $0<\delta\leq \delta_0$, $\Lambda^+(K,\delta)<\frac{2}{3}\Lambda^+(K)<0$. Since $\lim_{n\to+\infty}\frac{1}{n}\log (A_\delta(K,n))=\Lambda^+(K,\delta)$ there is $n_0\in\N$
such that for all $n\geq n_0=n_0(\delta_0)$, $\frac{1}{n}\log (A_{\delta_0}(K,n))<\frac{1}{2}\Lambda^+(K)<0$.
{\it A fortiori,} for all $n\geq n_0$ and $0<\delta\leq\delta_0$, $\frac{1}{n}\log (A_{\delta}(K,n))<\frac{1}{2}\Lambda^+(K)$ too.
Let us denote $-\gamma=\frac{1}{2}\Lambda^+(K)$.
Choose $\delta_0>\delta_1>0$ such that if $\dist(y,K)<\delta_1$ then $\dist(f^j(y),K)<\delta_0$ for all $j=0,1,2,\ldots, n_0$.
Finally let $U=\{y\in M\,:\, \dist(y,K)<\delta_1\}$.
If $y\in U$, since
$$\frac{1}{n}\log (A_\delta(K,n))=\frac{1}{n}\log \left(\sup_{y\in B^*_K(\delta,n)} \left\{\frac{\dist(K,f^n(y))}{\dist(K,y)}\right\}\right)< -\gamma $$
 we have that $\frac{\dist(K,f^n(y))}{\dist(K,y)}<e^{-\gamma n}$. But $\dist(y,K)<\delta_1<\delta_0$ and so
 $\dist(K,f^n(y))<e^{-\gamma n}\,\delta_0<\delta_0$ for all $n\geq n_0$ and we can apply induction. Thus $\dist(K,f^n(y))$ tends to zero when $n\to+\infty$ and $K$ is an attractor.

 The proof that $\lambda^-(K)>0$ implies that $K$ is a repeller is similar.

\end{proof}

\begin{tabbing}
Universidade Federal do Rio de Janeiro, \hspace{1cm}\= Facultad de Ingenieria, \kill
 M. J. Pacifico, \> J. L. Vieitez, \\
Instituto de Matematica, \> Departamento de Matematica,\\
Universidade Federal do Rio de Janeiro, \> CenUR Litoral Norte,\\
C. P. 68.530, CEP 21.945-970, \> Universidad de la Republica,\\
Rio de Janeiro, R. J. , Brazil. \> Rivera 1350 , CP 50000,\\
                                \> Salto, Uruguay\\
{\it pacifico@im.ufrj.br} \> {\it jvieitez@unorte.edu.uy}
\end{tabbing}

\end{document}